\documentclass[12pt]{amsart}

\usepackage{url,  mathrsfs}
\usepackage{hyperref}
\usepackage{amsmath}
\usepackage{graphicx, 
epstopdf}
\usepackage{subfig}
\usepackage{color}
\usepackage{bbm}
\usepackage{subfloat}
\usepackage[draft]{fixme}
\usepackage{bm} 
\usepackage{epigraph}
\usepackage{endnotes}
\usepackage{ifpdf}
\usepackage{array}

\usepackage{multicol}
\usepackage{multirow}
\usepackage{graphicx}
\usepackage{tikz}

\newcommand{\N}{\mathbb{N}}

\newcommand{\R}{\mathbb{R}}

\theoremstyle{plain}
\newtheorem{thm}{Theorem} 
\newtheorem{prop}[thm]{Proposition}

\newtheorem{lem}{Lemma}

\newtheorem{conjecture}[thm]{Conjecture}

\newtheorem*{thm*}{Theorem}

\theoremstyle{definition}
\newtheorem{Def}[thm]{Definition}
\newtheorem{exmp}[thm]{Example}

\theoremstyle{remark}

\usepackage[inner=1in,outer=1in,top=1in,bottom=1in]{geometry}

\title{A real stable extension of the V\'amos matroid polynomial }
\author{Sam Burton}
\address{Department of Mathematics, University of Michigan, Ann Arbor, USA}
\email{sburto@umich.edu}
\author{Cynthia Vinzant}
\address{Department of Mathematics, University of Michigan, Ann Arbor, USA}
\email{vinzant@umich.edu}
\author{Yewon Youm}
\address{Department of Mathematics, University of Michigan, Ann Arbor, USA}
\email{yyoum@umich.edu}
\date{\today}

\makeatletter
\newcommand{\addresseshere}{%
  \enddoc@text\let\enddoc@text\relax
}
\makeatother

\begin{document}

\thispagestyle{empty}

\maketitle
\bigskip
\begin{abstract}
In 2004, Choe, Oxley, Sokal and Wagner established a tight connection 
between matroids and multiaffine real stable polynomials. 
Recently,  Br{\"a}nd{\'e}n used this theory and a polynomial coming from the V\'amos matroid 
to disprove the generalized Lax conjecture. 
Here we present a 10-element extension of the V\'amos matroid 
and prove that its basis generating polynomial is real stable (\textit{i.e.}\ that 
the matroid has the half-plane property). We do this via
large sums of squares computations and a criterion for 
real stability given by Wagner and Wei. 
Like the V\'amos matroid, this matroid is not representable over any field
and no power of its basis generating polynomial can be written as the determinant of a
linear matrix with positive semidefinite Hermitian forms. 
\end{abstract}

\section{Introduction}

%
In 2004, Choe, Oxley, Sokal and Wagner \cite{COSW} showed that the 
support of any homogeneous multiaffine real stable polynomial 
is the set of bases of some matroid. 
One multiaffine stable polynomial that has been especially important in the
connections between this theory and other fields has been 
the basis generating polynomial of the V\'amos matroid $V_8$, given by
 \begin{equation}\label{eq:vamosPoly}
 f_8(x) = \sum_{B\in \binom{[8]}{4}\backslash H} \prod_{i\in B}x_i,  
 \end{equation}
 where $H = \{\{1,2,3,4\}, \{1,2,5,6\},\{1,2,7,8\},\{3,4,5,6\},\{5,6,7,8\}\}$. See Figure~\ref{fig:V8V10}.
 Wagner and Wei \cite{WW} showed that this polynomial is real stable. Using the non-representability of the V\'amos matroid, 
 Br{\"a}nd{\'e}n \cite{Bra11} showed that no power of $f_8$ can be written as the determinant of a
linear matrix with positive semidefinite Hermitian forms, thus disproving the generalized Lax conjecture. 
Afterwards this polynomial $f_8$ has been used as test case for many conjectures concerning determinantal representability of 
real stable polynomials \cite{Kummer, KPV, NPT}. 

The aim of this paper is to introduce a family of matroids $V_{2n}$ generalizing the V\'amos matroid and show that the next member, 
$V_{10}$, of this family corresponds to a real stable polynomial. The proof uses similar techniques to \cite{WW} and relies on large 
sums of squares computations. Our main theorem can be stated as the following:

\pagebreak

\begin{thm}\label{thm:main}
Let $\mathcal{B}$ consist of all subsets $B\subset \{1,\hdots, 10\}$ of size $|B|=4$ except for 
\[ \{1,2,3,4\},\{1,2,5,6\},\{1,2,7,8\},\{1,2,9,10\},\{3,4,5,6\},\{5,6,7,8\}, \text{ and }\{7,8,9,10\}. \]
Then $\mathcal{B}$ is the collection of the bases of a matroid and its basis-generating polynomial 
\[ f_{10}(x)\;\; = \;\;\sum_{B\in \mathcal{B}} \; \prod_{i\in B} x_i   \;\; \in \;\;\R[x_1, \hdots, x_{10}]_4\]
is stable. That is, the rank-four matroid $V_{10}$ with bases $\mathcal{B}$ has the half-plane property. 
Furthermore, like $f_8(x)$, no power of $f_{10}(x)$ has a definite determinantal representation. 
\end{thm}

\begin{figure}
\includegraphics[width=2in]{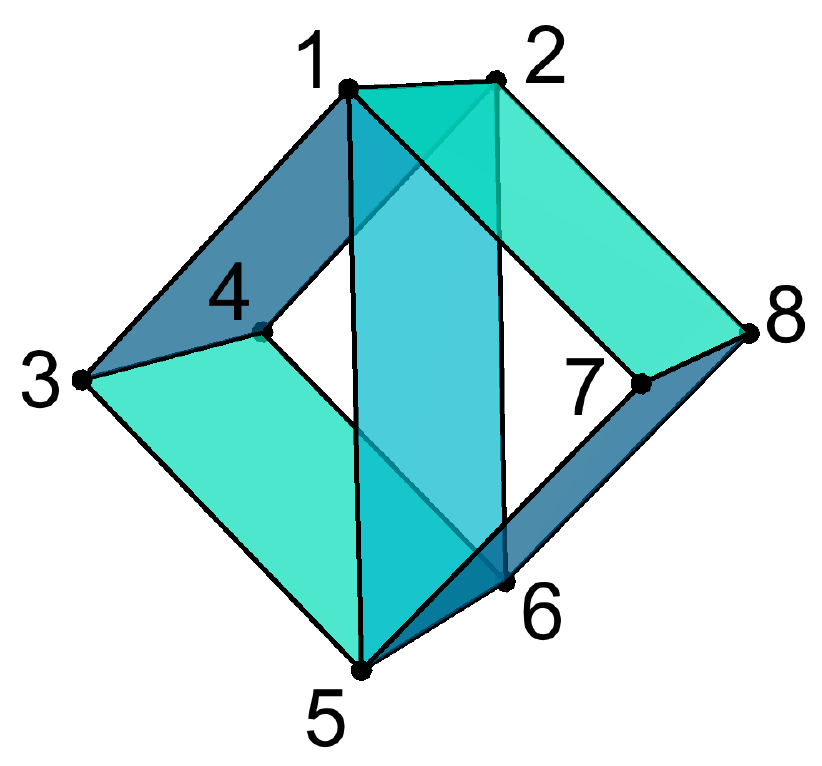} $\;\;\;\;\;\;\;\;\;\;$\includegraphics[width=2in]{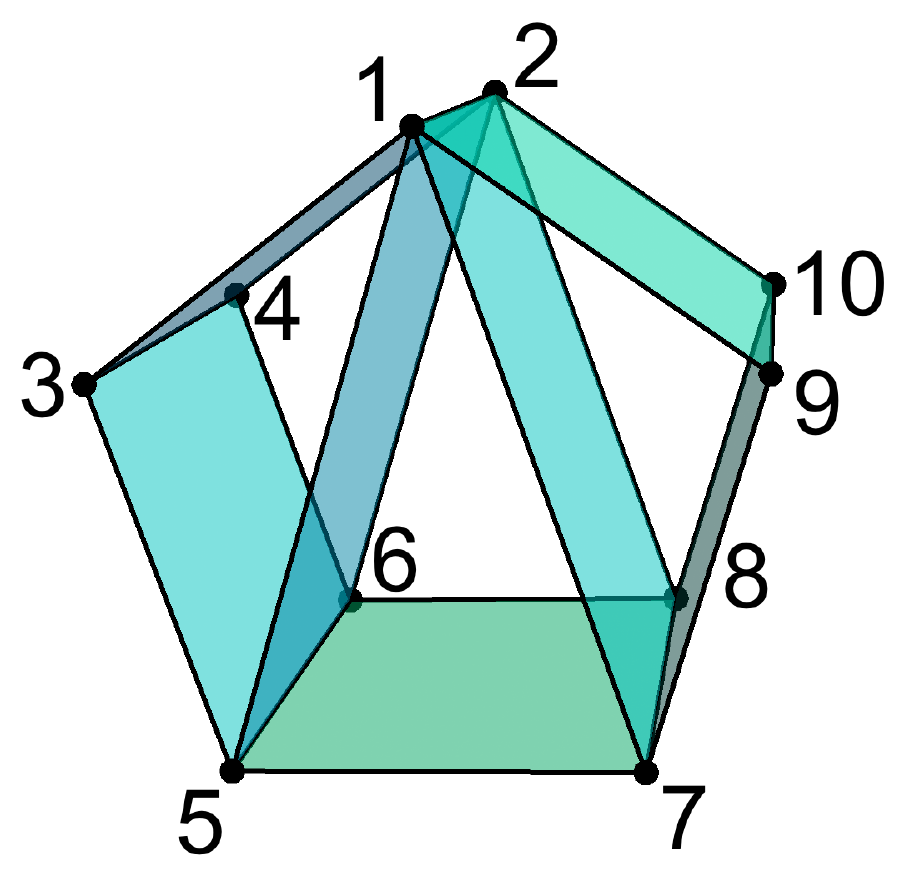} 
\caption{The rank-four V\'amos matroids $V_8$ and $V_{10}$.}\label{fig:V8V10}
\end{figure}

This paper is organized as follows. In Section~\ref{sec:background}, we review the necessary definitions and background on 
matroids and real stable polynomials.  We present a family of extensions $V_{2n}$ of the V\'amos matroid in Section~\ref{sec:extendV},
and in Section~\ref{sec:proof} we prove Theorem~\ref{thm:main}, namely that the basis generating polynomial of $V_{10}$ is stable.

\section{Background} \label{sec:background}

To begin, we go over some background material on matroid theory, 
multiaffine polynomials, and the relation between them.

\subsection{Matroid Basics}
Here we present some relevant definitions and concepts from the theory of matroids. 
For a more comprehensive take, we refer readers to  \cite{Oxley}.

\begin{Def}
A \textbf{matroid M} is an ordered pair $(E, \mathcal{I})$ consisting of a finite set $E$ and a collection $\mathcal{I}$ of subsets of $E$ having the following three properties: 
\begin{enumerate}
\item $\emptyset$ $\in$ $\mathcal{I}$.
\item If $I \in \mathcal{I}$ and $I' \subseteq I$, then $I' \in \mathcal{I}$.
\item If $I$ and $I'$ are in $\mathcal{I}$ and $|I| < |I'|$, then there is an element $e$ of $I'\backslash I$ such that $I \cup \{e\} \in \mathcal{I}$.
\end{enumerate}
The members of $\mathcal{I}$ are the \textbf{independent sets} of $M$, and $E$ is the \textbf{ground set} of $M$.
\end{Def}

A maximal independent set in $M$ is a \textbf{basis} of $M$. We denote the set of bases as $\mathcal{B}(M)$. 
All bases  have the same size, which is called the \textbf{rank} of $M$ and denoted $r(M)$. 
A matroid $(\{1,\hdots, n\}, \mathcal{I})$ is \textbf{representable} over a field $\mathbb{F}$ if there exists vectors 
$v_1, \hdots, v_n \in \mathbb{F}^r$ so that the independent sets $\mathcal{I}$ are exactly those sets $I$ for which 
the vectors $\{v_i: i\in I\}$ are independent. 
We say that two matroids $M = (E, \mathcal{I})$ and $M' = (E', \mathcal{I}')$ are \textbf{isomorphic} 
if there is a bijection $\pi:E\rightarrow E'$ 
for which $\mathcal{I}' = \{\pi(I) \;:\;I\in \mathcal{I} \}$.

Let $M$ be a matroid on $E$ and $\mathcal{B}^*(M)$ be the collection $\{E\backslash B : B \in \mathcal{B}(M)\}$. 
Then  $\mathcal{B}^*(M)$ is the set of bases of a matroid on $E$ of rank $|E|-r(M)$, which we call the \textbf{dual} matroid of $M$ and denote $M^*$. 
The \textbf{deletion} of $e\in E$ from matroid $M$ is the matroid $M \backslash e$ on $E\backslash\{e \}$ with bases 
$\{B : B \in \mathcal{B} \text{ with } e \not\in B\}.$ The \textbf{contraction} of $M$ by is given by $M/e = (M^*\backslash e)^*.$

We picture a rank-$r$ matroid $M$ as a collection of points in $(r-1)$-space and draw a hyperplane through collections 
of $r$ elements that are dependent, \textit{i.e.}\ not bases of $M$. 

\begin{exmp}\label{ex:repMatroid}
Consider the rank-3 matroid $M(A)$ represented by the columns of the matrix
$$A  \ = \  
\begin{pmatrix}
v_1 & v_2 & v_3 & v_4 & v_5
\end{pmatrix}
\ = \ 
\begin{pmatrix}
1 & 0 & 0 & 1 & 0\\
0 & 1 & 0 & 1 & 1\\
0 & 0 & 1 & 0 & 1
\end{pmatrix}.$$
Every size-three subset of $\{1,2,3,4,5\}$ is a basis of $M=M(A)$ except for 
$\{1, 2, 4\}$ and $\{2, 3, 5\}$.
Its deletion at $\{1\}$ is represented by the columns of the $3\times 4$ matrix obtained by deleting the first column of $A$.  
The bases of $M\backslash 1$ are precisely those in $\mathcal{B}(M)$ that do not contain the element $1$. That is, $\mathcal{B}(M\backslash 1) = \{ \{2,3,4\}, \{2,4,5\}, \{3,4,5\} \} $.
See Figure~\ref{fig:ex}.

The bases of the contraction $M/1$ are the bases of $M$ that contained the element $1$ but with the $1$ now removed. That is, 
$\mathcal{B}(M/1) = \{ \{2, 3\}, \{2, 5\}, \{3, 4\}, \{3, 5\}, \{4, 5\} \}$. 
The matroid $M/1$ is represented by the columns of the $2\times 4$ matrix obtained by deleting the first column $v_1$ of $A$ and projecting the others onto $v_1^\perp$.
This represents both $M\backslash 1$ and $M/1$:
$$
M\backslash 1 \ \cong \ 
M\!\begin{pmatrix}
0 & 0 & 1 & 0\\
1 & 0 & 1 & 1\\
0 & 1 & 0 & 1
\end{pmatrix}
\ \  \ \text{ and } \ \ \ 
M/1 \  \cong \  
M\! \begin{pmatrix}
1 & 0 & 1 & 1\\
0 & 1 & 0 & 1
\end{pmatrix}. $$
  \hfill $\diamond$
\end{exmp}

\begin{center}
\begin{figure}\begin{center}
\includegraphics[height=1in]{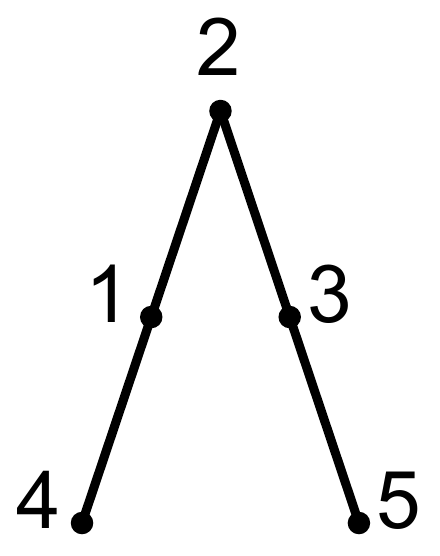} $\;\;\;\;\;\;\;\;\;\;\;\;\;\;$ \includegraphics[height=1in]{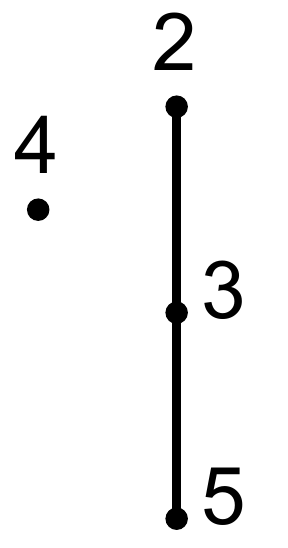} $\;\;\;\;\;\;\;\;\;\;\;\;\;\;$ \includegraphics[height=1in]{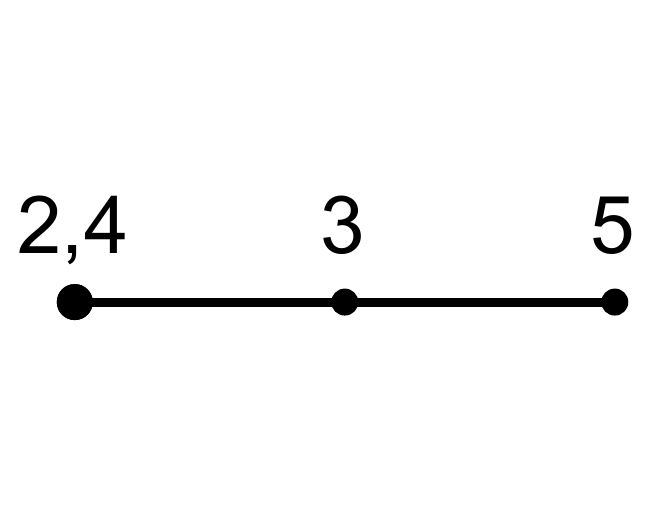}
 \end{center}
\caption{The matroid of Example~\ref{ex:repMatroid} and its deletion and contraction at $\{1\}$.}\label{fig:ex}
\end{figure}
\end{center}

Another important example is the matroid represented by $n$ \textit{generic} vectors in $\R^r$.

\begin{exmp}
For any $r\leq n\in\N$, the \textbf{uniform matroid} of rank $r$ on $n$ elements, $U_{r,n}$, is the matroid on $\{1,\hdots, n\}$
with independent sets $\mathcal{I}=\{I\subset E, |I|\leq r\}$.  It indeed has rank $r$ and its bases are the subsets of $\{1,\hdots, n\}$ 
of size exactly $r$.  Its deletion at any element $e$ is isomorphic to $U_{r, n-1}$ and the contraction $U_{r,n}/e$ is isomorphic to $U_{r-1,n-1}$.
 \hfill $\diamond$
\end{exmp}

\subsection{Polynomials and real stability}
To a matroid $M = (\{1,2,\hdots, n\}, \mathcal{I})$, we can associate its \textbf{basis generating polynomial}
$$f(M)\;\;=\;\;\sum_{B\in\mathcal{B}(M)}\prod_{i\in B}x_i \;\;\;\in\;\; \R[x_1, \hdots, x_n]. $$
This polynomial is homogeneous with degree equal to the rank of $M$.  It is also \textbf{multiaffine}, meaning that every monomial
is square free.   The operations of deletion and contraction correspond nicely to operations on the level of polynomials, 
namely restricting a variable to zero and taking a partial derivative. That is, for any $i\in M$. 
\[f(M\backslash i)=(f(M))|_{x_i=0} \ \ \  \text{and} \ \ \ f(M/i)=\frac{\partial}{\partial x_i}f(M).\]
We are particularly interested in when these polynomials are \emph{stable}.

\begin{Def}
Let $f$ be a real homogeneous polynomial in $\R[x_1,\dots,x_n]$. We say that $f$ is 
 \textbf{stable} if for every choice of real vectors $v \in (\R_{>0})^n$ and $w \in \R^n$, the univariate polynomial 
 $f(tv+w) \in \mathbb{R}[t]$ has only real roots. Equivalently, one can say that $f$ 
has the \textbf{half-plane property} or that $f$ is \textbf{strongly Rayleigh} \cite{Bra07}.
 
We say that a matroid $M$ has the \textbf{half-plane property} if the multiaffine polynomial $f(M)$ does. So a matroid $M$ has the half-plane property if and only if $f(M)$ is stable.\\
 \end{Def}

Choe, Oxley, Sokal, and Wagner \cite{COSW} show that if $f\in \R[x_1, \hdots, x_n]$ is homogeneous, multiaffine, and stable, 
then its \emph{support} (the collection of $I\subset \{1,\dots,n\}$ for which
the monomial $\prod_{i\in I}x_i$ appears in $f$) is the set of bases
of a matroid. One important example of this phenomenon 
comes from representable matroids. 

\begin{exmp} \label{ex:detrep}
Given positive semidefinite matrices $A_1, \hdots, A_n \in \R^{d\times d}_{\rm sym}$, 
the polynomial $f(x) = \det(x_1A_1+ \hdots+ x_nA_n )$ is real stable (if it is not identically zero).  To see this, let $A(x)$ denote $\sum_ix_iA_i$.
For any $v\in (\R_{>0})^n$ and $w \in \R^n$, the matrix $A(v)$ will be positive definite and we can write $A(v) = UU^T$ for 
some real matrix $U$. Then 
\[f(tv+w) \ = \ \det(tA(v)+ A(w)) \ = \ \det(U)^2\det(tI + U^{-1}A(w)U^{-T}).\]
For any $w\in \R^n$, the real symmetric matrix $U^{-1}A(w)U^{-T}$ has all real eigenvalues and thus all the roots of the polynomial $f(tv+w)$ are real. 
Furthermore, if the matrices $A_i$ have rank-one, that is $A_i = v_iv_i^T$ for some vectors $v_i\in \R^d$, then 
$\det(\sum_i x_iA_i)$  can be written as \[\det\bigl(\sum_i x_i v_iv_i^T\bigl)  \ = \ \sum_{I\in\binom{[n]}{d}} \det(v_i:i\in I)^2\prod_{i\in I}x_i.\]
This polynomial is multiaffine and its support will be the bases of the matroid represented by the the vectors $v_1, \hdots, v_n$. \hfill $\diamond$
\end{exmp}

One can also obtain real stable polynomials by taking directional derivatives. 
For a homogeneous, stable polynomial $f \in \R[x_1, \hdots, x_n]$ and any $\lambda\in (\R_{\geq 0})^n$, the
polynomial $\sum_i \lambda_i \partial f/\partial x_i$ is stable as well \cite[Prop. 2.8]{COSW}.

\begin{exmp}\label{ex:uniform}
For $r\leq n$, let $e_{r,n}(x)$ be the $r$th elementary symmetric function 
$\sum_{I\in \binom{[n]}{r}}\prod_{i\in I}x_i$. 
Clearly $e_{n,n}(x) = x_1\cdots x_n$ is stable. This implies that $e_{r,n}(x) = (n-r)\sum_{i=1}^n \partial e_{r+1,n}/\partial x_i$
is stable for all $r$.  Thus the uniform matroid $U_{r,n}$, whose basis generating polynomial is $e_{r,n}$, 
has the half-plane property. \hfill $\diamond$
\end{exmp}

In general, showing the real stability of a polynomial is difficult and not all stable polynomials come from determinants or derivatives. 
Similarly, the relationship between representability of a matroid and half-plane property is not very well understood. 
The V\'amos matroid $V_8$ has the half-plane property but is not representable over any field. 
Moreover, not every matroid is the support of a real stable polynomial \cite[Theorem 6.6]{Bra07}.  
For a large database of matroids and their status, see \cite[Table 2.1]{GL}. 

In \cite{Bra07}, Br{\"a}nd{\'e}n gives a criterion for the real stability of multiaffine polynomials (and hence the half-plane property of matroid) 
using its Rayleigh differences. The \textbf{Rayleigh difference} of a polynomial $f$ at $i, j \in [n]$ is
\begin{equation}\label{eq:Delta}
\Delta_{i,j}f \ \ = \ \ \frac{\partial f}{\partial x_i} \cdot \frac{\partial f}{\partial x_j}-f \cdot \frac{\partial^2 f}{\partial x_i\partial x_j}.
\end{equation}
An extension of Br{\"a}nd{\'e}n's work is following characterization of stability in terms of Rayleigh differences, which appears as Theorem 3.1 of \cite{WS}.

\begin{thm}\label{thm:HPP} 
Let $f\in\R[x_1,\dots,x_n]$ be a multiaffine polynomial with positive coefficients. The following are equivalent. \begin{enumerate}
\item $f$ is stable. \smallskip
\item For all $\{i,j\}\subset\{1,\dots,n\}$ and for all $a\in\R^{n-2}$, $\Delta_{i,j}f(a)\geq0$.  \smallskip
\item Either $n=1$ or there exist some pair of indices $\{i,j\}\subset\{1,\dots,n\}$ such that $f|_{x_i=0}$, $\frac{\partial}{\partial x_i}f$, $f|_{x_j=0}$, and $\frac{\partial}{\partial x_j}f$ are stable and for all $a\in\R^{n-2}$, $\Delta_{i,j}f(a)\geq0$. 
\label{it}
\end{enumerate} 
\end{thm}
In the proof of Theorem~\ref{thm:main}, we will use the equivalence of (1) and (3) to build up a collection 
of stable polynomials, culminating in our desired polynomial $f_{10}$.

\section{Extended V\'amos Matroids} \label{sec:extendV}

The V\'amos matroid $V_8$ is the rank-four matroid on eight elements whose bases consist of all 4-element subsets of $\{1,\hdots, 8\}$
except for  $\{1,2,3,4\}$, $\{1,2,5,6\}$, $\{1,2,7,8\}$, $\{3,4,5,6\}$, and $\{5,6,7,8\}$.  This matroid is not representable over any field \cite[Example 2.1.25]{Oxley}. 
Wagner and Wei \cite{WW} showed that $V_8$ has the half-plane property, meaning that $f_8$, its basis generating polynomial  \eqref{eq:vamosPoly}, is real stable.  In 2011, Br{\"a}nd{\'e}n \cite{Bra11} showed that no power of $f_8$ has a \emph{definite determinantal
 representation}. That is, for every $k\geq 0$, the polynomial $(f_8)^k$ cannot be written as $\det(x_1A_1+\hdots + x_8A_8)$ with 
positive semidefinite matrices $A_1, \hdots, A_8$. This was the first such example and disproved a conjecture of Helton and Vinnikov \cite{HV}. 

In order to show the real stability of $f_8$, Wagner and Wei use Theorem~\ref{thm:HPP} and write 
the Rayleigh difference $\Delta_{1,3}(f_8)$ as a sum of squares \cite{WW}. On the other hand, the Rayleigh difference 
$\Delta_{7,8}(f_8)$ is nonnegative but not a sum of squares. This can be used to verify that no power of 
$f_8$ has a definite determinantal representation \cite[Ex 5.11]{KPV}. 

Interestingly, there is evidence \cite{Kummer} that the polynomial $f_8$ does not provide a counterexample for a weaker generalization of 
the Lax conjecture. To investigate such conjectures further, it would be useful to have more real stable polynomials for which no power 
has a definite determinantal representation. We propose the following family of matroid polynomials. 

\begin{Def}
For $n\geq 4$, let $\mathcal{H}_{2n}$ denote the collection of subsets of $\{1,\hdots, 2n\}$ of the form  
\[\{1,2, 2k-1,2k\}\text{ or }\{2k-1,2k, 2k+1, 2k+2\}\text{ for }2 \leq k\leq n.\]  We define the \textbf{$2n$-V\'amos matroid}, denoted $V_{2n}$, 
to be the rank-four matroid whose bases are all subsets $B\subset \{1,\hdots, 2n\}$ with $|B|=4$ except for $B\in \mathcal{H}_{2n}$. 
See Figure~\ref{fig:V8V10}. 
We call its basis generating polynomial $f_{2n} = f(V_{2n})$. 
\end{Def}

\begin{prop}
The $2n$-V\'amos matroid defined above is indeed a matroid.
\end{prop}
\begin{proof}
We need a little more matroid notation. 
A \textbf{hyperplane} of a matroid $M = (E,\mathcal{I})$ is a subset $H\subset E$ of rank $r(M)-1$ such that for any 
$i\in E\backslash H$, the rank of $E\cup \{i\}$ equals $r(M)$. An \textbf{$m$-partition} on the set $E$ is a collection of subsets 
$\mathcal{T}=\{T_1,\dots,T_k\}$ such that each member of $\mathcal{T}$ has at least $m$ elements and each $m$ 
element subset of $E$ is contained in a unique $T_i$. Proposition 2.1.24 of \cite{Oxley} states that 
any $m$-partition of a set $E$ is the collection of hyperplanes of a rank-$(m+1)$ matroid on $E$.

Consider the union of $\mathcal{H}_{2n}$ and the collection of subsets $I \subset[n]$ of size three that are not contained in 
any element of $\mathcal{H}_{2n}$. This union forms a $3$-partition of $[n]$ and thus it is the collection of hyperplanes of a rank-4 matroid. 
This matroid is exactly $V_{2n}$. 
\end{proof}

Every $2n$-V\'amos matroid has a minor isomorphic to $V_8$, so for $n\geq4$, $V_{2n}$ is non-representable over any field. 
Similarly, for $n\geq 4$ the polynomials $f_{2n}$ all have restrictions equal to $f_8$, 
and so no power of $f_{2n}$ has a definite determinantal representation. We hope that $\{f_{2n}\}$ gives a family 
of real stable polynomials with this property. 

\begin{conjecture} \label{conj:HPP}
For every $n \geq 4$, the matroid $V_{2n}$ has the half-plane property.
\end{conjecture}

This holds for $n=4$ (above) and $n=5$ (Section~\ref{sec:proof}). 
Numerical evidence suggests that $V_{12}$ also has the half-plane property but 
larger computations become intractable. One promising approach is to use more refined sums of squares techniques, e.g. \cite[Ex.7.5.2]{FrankP}.

\section{Proof of Theorem~\ref{thm:main}} \label{sec:proof}

As discussed in Section~\ref{sec:extendV}, $V_{2n}$ defines a matroid and no power of its basis generating polynomial $f_{2n}$ 
has a definite determinantal representation. The remaining content of Theorem~\ref{thm:main} is that the polynomial $f_{10}$ is stable, 
or equivalently, the matroid $V_{10}$ has the half-plane property, which we show in Lemma~\ref{victory} below.  
To prove this we rely heavily on Theorem~\ref{thm:HPP} and inductively show that 
certain derivatives and restrictions are stable. Recall that for a matroid $M$
and its basis generating polynomial $f(M)$, derivation corresponds to matroid contraction 
($\frac{\partial}{\partial x_i}f(M) = f(M/i)$) and restriction corresponds to matroid deletion ($f(M)|_{x_i=0} = f(M\backslash i))$.
We will show that certain deletions and contractions of $V_{10}$ have the half-plane property.

We build off of known results for smaller matroids. 
Choe \textit{et.\ al.}\ show that the following matroids have the half-plane property: 
\begin{itemize}
\item any matroid of rank $2$ \cite[Corollary 5.5]{COSW},\smallskip
\item any uniform matroid $U_{r,n}$, as in Example~\ref{ex:uniform} or \cite[Section~9]{COSW},\smallskip
\item the 7-element matroids $F_{7}^{-5}$,  $F_{7}^{-6}$, and its dual $(F_{7}^{-6})^*$ \cite[A.2.2]{COSW}. See Figure~\ref{fig:Fano}.
\end{itemize}

\begin{center}
\begin{figure}[h]
\begin{center}
\begin{tabular}{ccc}
\includegraphics[height=1.2in]{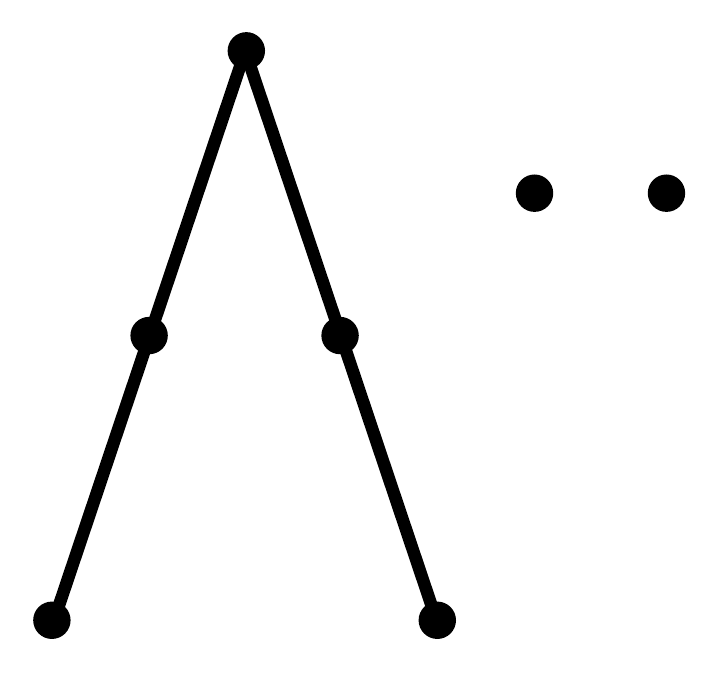}$\;\;\;$ &$\;\;\;$  \includegraphics[height=1.2in]{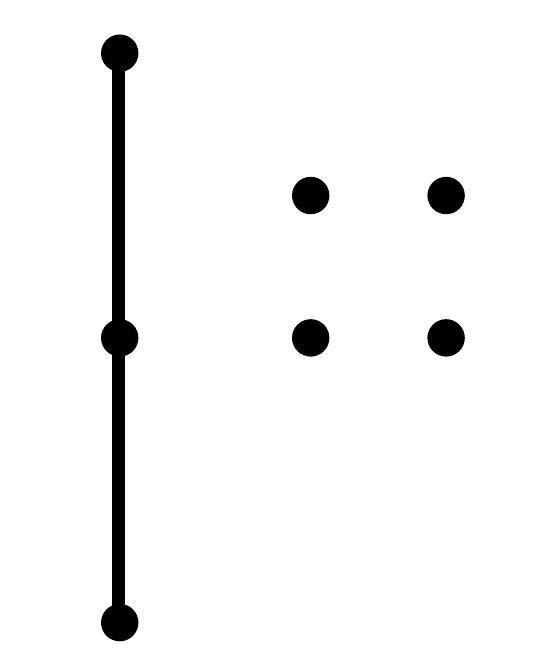} $\;\;\;$ &$\;\;\;$  \includegraphics[height=1.2in]{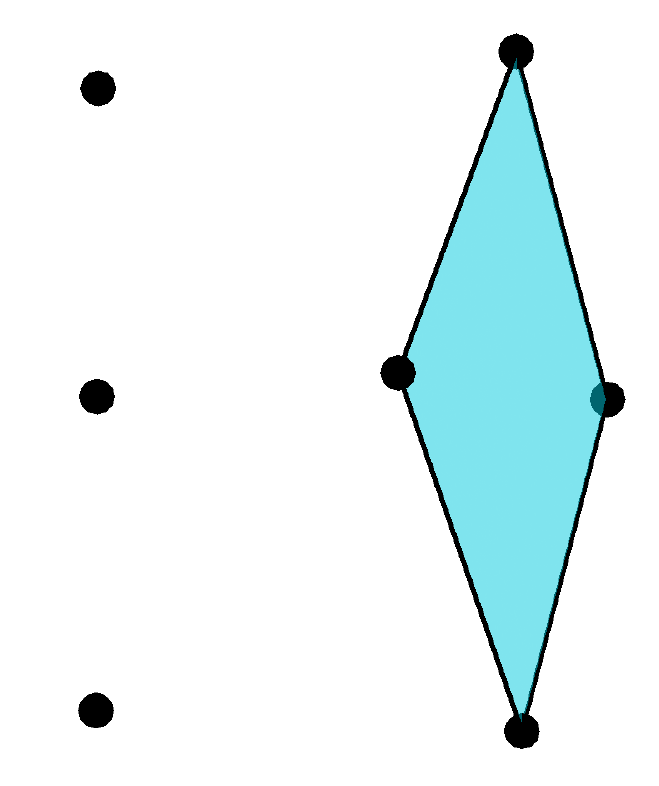}  \medskip \\ 
$F_7^{-5}$ $\;\;\;$ &$\;\;\;$  $F_7^{-6}$ $\;\;\;$ &$\;\;\;$  $(F_7^{-6})^*$
\end{tabular}
 \end{center}
\caption{Some 7-element matroids with the half-plane property.}\label{fig:Fano}
\end{figure}
\end{center}

The other main component of the proof is using Theorem~$\ref{thm:HPP}$ and showing that certain Rayleigh differences \eqref{eq:Delta} are nonnegative. 
In order to do this we write them as sums of squares. In particular, we find positive semidefinite \textit{Gram} matrices for these polynomials.  
For a polynomial $f \in \R[x_1, \hdots, x_n]$, a \textbf{p.s.d.\ Gram representation} of $f$ 
is a pair $(m,G)$ of a vector $m$ of monomials in $x_1, \hdots, x_n$
and a real symmetric positive semidefinite matrix $G$ for which $f = m^TGm$. 

A p.s.d.\ Gram representation of $f$ corresponds to a representation of $f$ as a \textit{sum of squares} and in particular 
implies that $f$ is nonnegative on $\R^n$. To see this, note that we can factor the p.s.d.\ matrix as as sum of real rank-one 
matrices $G = \sum_i g_i g_i^T$. This writes $f$ as a sum of the squares of the polynomials $g_i^Tm$:
\[ f \ \ =\ \  m^TGm \ \ = \ \  \sum_i m^T g_i g_i^T m \ \ = \ \  \sum_i (g_i^Tm)^2. \]
After fixing the monomials $m$, the search for a positive semidefinite Gram matrix $G$ is a \textit{semidefinite program} and can be carried out 
numerically in {\tt Matlab}. In order to find the p.s.d.\ Gram matrices in the appendix, we used the sums-of-squares optimization package 
SOSTOOLS \cite{sostools} in YALMIP \cite{YALMIP} and did some ad hoc rationalization in {\tt Mathematica}.  
We verified that all of the resulting rational matrices are positive semidefinite by calculating (in exact arithmetic) their diagonal minors.  
These computations are available at \url{http://www-personal.umich.edu/~vinzant/SosVamosRayleighDifferences.nb}.

By finding p.s.d.\ Gram representations of various Rayleigh differences, we build up 
a collection of matroids with the half-plane property, ending with $V_{10}$. See Figure~\ref{fig:Lemmata}.

\begin{center}
\begin{figure}[b]
\begin{center}
\begin{tabular}{cc}
 \includegraphics[height=1.7in]{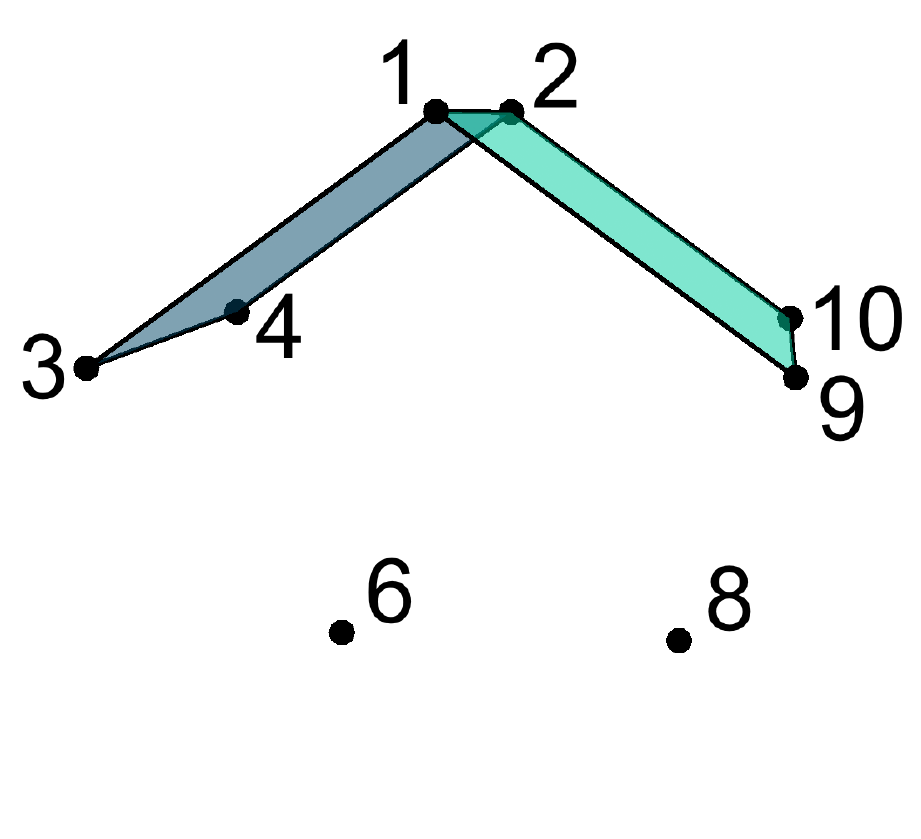}$\;\;\;\;\;$&$\;\;\;\;\;$\includegraphics[height=1.7in]{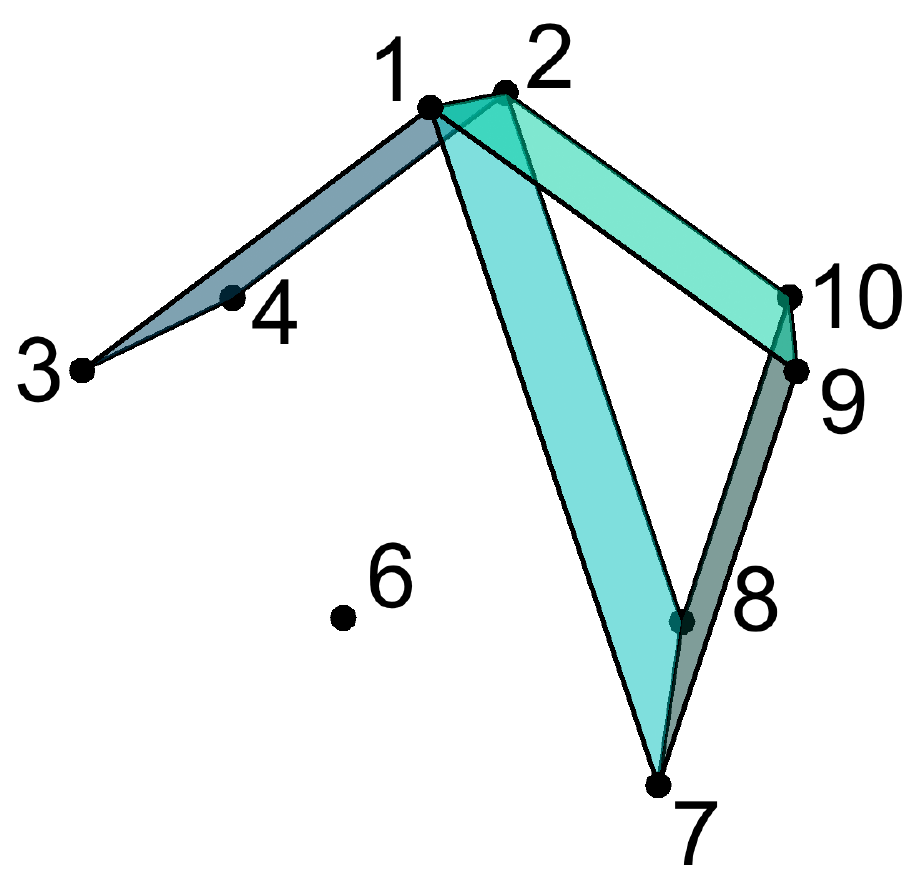} \\
 $V_{10}\backslash\{5,7\}$$\;\;\;\;\;$&$\;\;\;\;\;$ $V_{10}\backslash 5$  \bigskip\\
&\\
\includegraphics[height=1.3in]{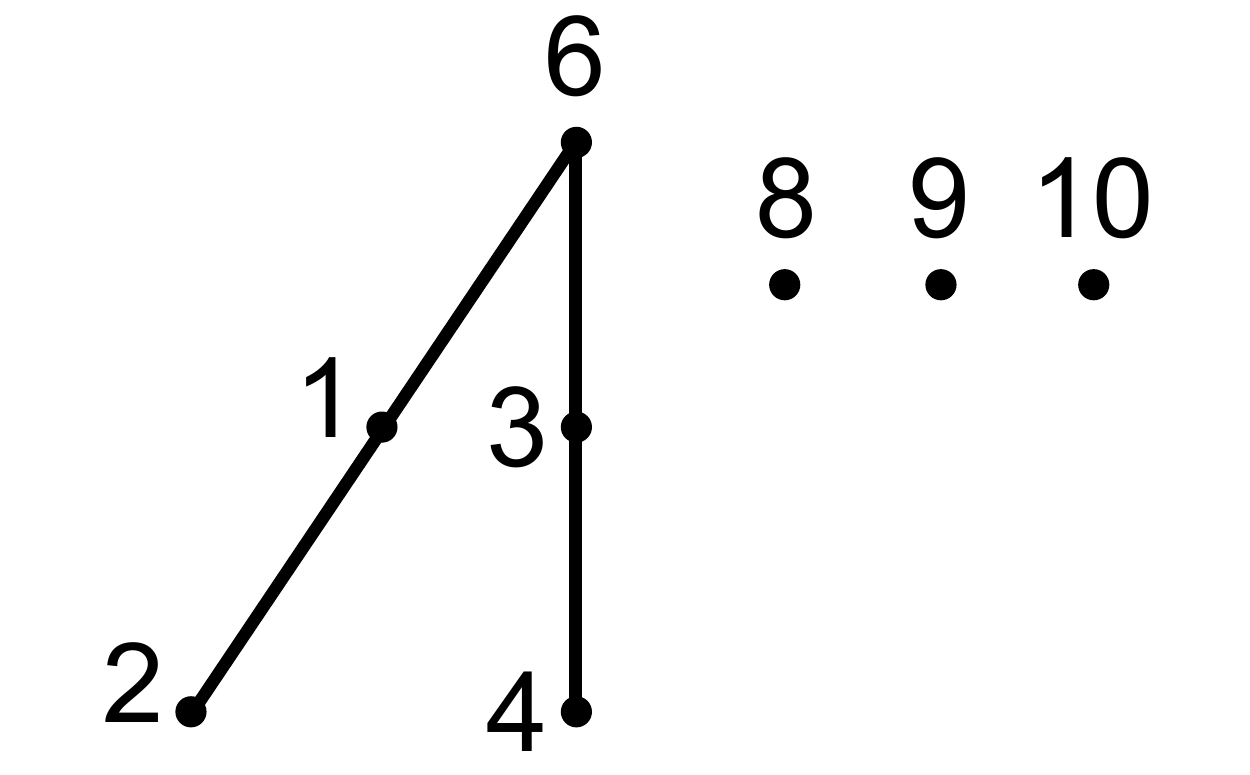} $\;\;\;\;\;$&$\;\;\;\;\;$ \includegraphics[height=1.3in]{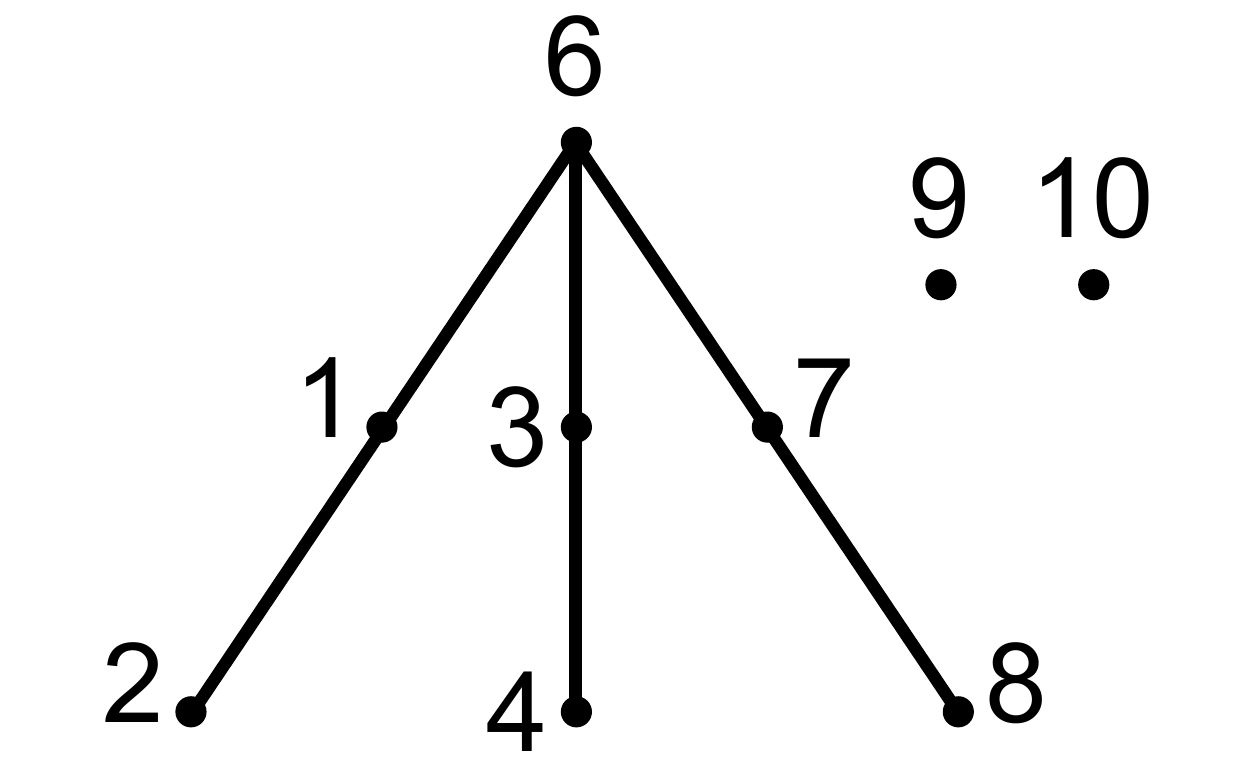}   \medskip \\
$V_{10}/5\backslash 7$  $\;\;\;\;\;$&$\;\;\;\;\;$ $V_{10}/5$\\
 \end{tabular}
 \end{center}
\caption{Some deletions and contractions of $V_{10}$ with the half-plane property.}  \label{fig:Lemmata}
\end{figure}
\end{center}

\pagebreak 
\begin{lem}\label{twoplanes}The matroid $V_{10}\backslash \{5,7\}$ has the half-plane property.\end{lem}
\proof 
The contraction of $V_{10} \backslash \{5,7\}$ by $\{1\}$ is isomorphic to the matroid $F_7^{-5}$ and its
contraction by $\{3\}$ is isomorphic to the matroid $F_7^{-6}$, both of which have the half-plane property. 
On the other hand, the deletion of $V_{10} \backslash \{5,7\}$ at $\{1\}$ is isomorphic to the uniform matroid $U_{4,7}$
and the deletion of $V_{10} \backslash \{5,7\}$ at $\{3\}$ is isomorphic to the matroid $(F_7^{-6})^*$, which also have the half-plane property. 
In the appendix, we give a p.s.d.\ Gram representation $(m_1,G_1)$ of the Rayleigh difference $\Delta_{1,3}f(V_{10} \backslash \{5,7\})$, showing it to be nonnegative on $\R^n$. 
Hence, by Theorem~$\ref{thm:HPP}$, $f(V_{10}\backslash \{5,7\}) = f_{10}|_{x_5=x_7=0}$ is stable.  
\endproof

\begin{lem}\label{C58}The matroid $V_{10}/5\backslash7$ has the half-plane property.\end{lem}
 \proof All one-point contractions of $V_{10}/5\backslash7$ have rank 2 and therefore the half-plane property.
 The deletion of $V_{10}/5\backslash7$ at $\{1\}$ is isomorphic to $F_7^{-6}$ and its deletion at $\{6\}$ is isomorphic to the
 uniform matroid $U_{3,7}$, both of which have the half-plane property.  In the appendix, 
 we give a p.s.d.\ Gram representation $(m_2,G_2)$ of  $\Delta_{1,6}f(V_{10}/5\backslash7)$, showing it to be nonnegative on $\R^n$. 
  Thus by Theorem~$\ref{thm:HPP}$, $f(V_{10}/5\backslash7) = \frac{\partial}{\partial x_5}f_{10}|_{x_7=0}$ is stable. \endproof

\begin{lem}\label{C79}The matroid $V_{10}/5$ $(\cong V_{10}/7)$ has the half-plane property.\end{lem}
 \proof All one-point contractions of $V_{10}/5$ have rank $2$ and so have the half-plane property. 
 The one point deletion $V_{10}/5\backslash1$ is isomorphic to the matroid $V_{10}/5\backslash7$, and both have the half-plane property by Lemma $\ref{C58}$. 
 We write a p.s.d.\ Gram representation $(m_3,G_3)$ of the  Rayleigh difference $\Delta_{1,7}f(V_{10}/5)$ in the appendix, which proves it to be nonnegative. 
Using Theorem $\ref{thm:HPP}$ with $i=1$, $j=7$, we see that $f(V_{10}/5) = \frac{\partial}{\partial x_5}f_{10}$ is stable.  
 \endproof

\begin{lem}\label{planeprism}The matroid $V_{10}\backslash 5$ $( \cong V_{10}\backslash 7)$ has the half-plane property.\end{lem}
 \proof The one-point contraction $V_{10}\backslash 5/9$ is isomorphic to $V_{10}\backslash5/7$, and both the half-plane property by Lemma $\ref{C58}$. 
 The one point deletion $V_{10}\backslash\{5,9\}$ is isomorphic to $V_{10}\backslash\{5,7\}$, and both have the half-plane property by Lemma $\ref{twoplanes}$. 
The Rayleigh difference $\Delta_{7,9}f(V_{10}\backslash 5)$ has a p.s.d.\ Gram representation $(m_4,G_4)$ in the appendix, which shows it to be nonnegative. 
 Using Theorem $\ref{thm:HPP}$ with $i=7$, $j=9$, we see that $f(V_{10}\backslash 5) = f_{10}|_{x_5=0}$ is stable and $V_{10}\backslash 5$ has the half-plane property.  
 \endproof

\begin{lem}\label{victory}The matroid $V_{10}$ has the half-plane property.\end{lem}
\proof The one-point contractions of $V_{10}$ at $5$ and $7$ have the half-plane property by Lemma~\ref{C79}. The one point deletions $V_{10}\backslash5 $ and $V_{10}\backslash7$ 
have the half-plane property by Lemma~\ref{planeprism}.  Moreover, the pair $(m_5,G_5)$ in the appendix is a p.s.d.\ Gram factorization of the Rayleigh difference
$\Delta_{5,7}f_{V_{10}}$. Thus $\Delta_{5,7}f_{V_{10}}$ is a sum of squares and certifiably nonnegative. 
Again using Theorem~$\ref{thm:HPP}$ with $i=5$, $j=7$, we see that $f_{10}$ is stable.  \endproof   

This concludes the proof of Theorem~\ref{thm:main}. The representations below are not unique and there may be other Gram representations 
that are more conducive to generalization. 
It would be interesting to find a family of p.s.d.\ Gram representations for Rayleigh differences that could be used to prove Conjecture~\ref{conj:HPP}.

\bigskip

\textbf{Acknowledgements.} We would like to thank David Speyer and Frank Permenter for helpful discussions. 
This paper started as a summer project for the University of Michigan REU Program, which 
supported Sam Burton and Yewon Youm. Cynthia Vinzant was supported by an NSF postdoc DMS-1204447.

\bibliography{VamosRefs}{}
\bibliographystyle{plain}

\vfill

\addresseshere

\appendix
\section*{Appendix: P.s.d.\ Gram representations of Rayleigh differences}

A p.s.d.\ Gram representation of $\Delta_{1,3}(f(V_{10}\backslash\{5,7\})=\Delta_{1,3}(f_{10}|_{x_5=x_7=0})$ is $ m_1,G_1:$
\[
\left(
\begin{array}{c}
 x_6 x_8 x_{10} \\
 x_4 x_8 x_{10} \\
 x_4 x_6 x_{10} \\
 x_4 x_6 x_8 \\
 x_2 x_8 x_{10} \\
 x_2 x_6 x_{10} \\
 x_2 x_6 x_8 \\
 x_2 x_4 x_{10} \\
 x_2 x_4 x_8 \\
 x_2 x_4 x_6 \\
 x_8 x_9 x_{10} \\
 x_6 x_9 x_{10} \\
 x_6 x_8 x_9 \\
 x_4 x_9 x_{10} \\
 x_4 x_8 x_9 \\
 x_4 x_6 x_9 \\
 x_2 x_8 x_9 \\
 x_2 x_6 x_9 \\
 x_2 x_4 x_9 \\
\end{array}
\right)
\ \
{\small \left(
\begin{array}{ccccccccccccccccccc}
 1 & \frac{1}{2} & \frac{1}{2} & \frac{1}{2} & \frac{1}{2} & \frac{1}{2} & \frac{1}{2} & \frac{1}{3} & \frac{1}{3} & \frac{1}{3} & \frac{1}{2} &
   \frac{1}{2} & \frac{1}{2} & \frac{1}{3} & \frac{1}{3} & \frac{1}{3} & \frac{1}{3} & \frac{1}{3} & \frac{1}{3} \medskip \\ 
 \frac{1}{2} & 1 & \frac{1}{2} & \frac{1}{2} & \frac{1}{2} & \frac{1}{3} & \frac{1}{3} & \frac{1}{2} & \frac{1}{2} & \frac{1}{2} & \frac{1}{2} &
   \frac{1}{3} & \frac{1}{3} & \frac{1}{2} & \frac{1}{2} & \frac{1}{3} & \frac{1}{3} & \frac{1}{3} & \frac{1}{2} \medskip \\ 
 \frac{1}{2} & \frac{1}{2} & 1 & \frac{1}{2} & \frac{1}{3} & \frac{1}{2} & \frac{1}{3} & \frac{1}{2} & \frac{1}{2} & \frac{1}{2} & \frac{1}{3} &
   \frac{1}{2} & \frac{1}{3} & \frac{1}{2} & \frac{1}{3} & \frac{1}{2} & \frac{1}{3} & \frac{1}{3} & \frac{1}{2} \medskip \\
 \frac{1}{2} & \frac{1}{2} & \frac{1}{2} & 1 & \frac{1}{3} & \frac{1}{3} & \frac{1}{2} & \frac{1}{2} & \frac{1}{2} & \frac{1}{2} & \frac{1}{3} &
   \frac{1}{3} & \frac{1}{2} & \frac{1}{3} & \frac{1}{2} & \frac{1}{2} & \frac{1}{3} & \frac{1}{3} & \frac{1}{2} \medskip\\
 \frac{1}{2} & \frac{1}{2} & \frac{1}{3} & \frac{1}{3} & 1 & \frac{1}{2} & \frac{1}{2} & \frac{1}{2} & \frac{1}{2} & \frac{1}{2} & \frac{1}{2} &
   \frac{1}{4} & \frac{1}{3} & \frac{1}{6} & \frac{1}{3} & \frac{1}{3} & \frac{1}{2} & \frac{1}{4} & \frac{1}{2} \medskip \\
 \frac{1}{2} & \frac{1}{3} & \frac{1}{2} & \frac{1}{3} & \frac{1}{2} & 1 & \frac{1}{2} & \frac{1}{2} & \frac{1}{2} & \frac{1}{2} & \frac{1}{4} &
   \frac{1}{2} & \frac{1}{3} & \frac{1}{6} & \frac{1}{3} & \frac{1}{3} & \frac{1}{4} & \frac{1}{2} & \frac{1}{2} \medskip\\
 \frac{1}{2} & \frac{1}{3} & \frac{1}{3} & \frac{1}{2} & \frac{1}{2} & \frac{1}{2} & 1 & \frac{1}{2} & \frac{1}{2} & \frac{1}{2} & \frac{1}{3} &
   \frac{1}{3} & \frac{1}{2} & -\frac{1}{6} & \frac{1}{3} & \frac{1}{3} & \frac{1}{2} & \frac{1}{2} & \frac{1}{2} \medskip\\
\end{array}
\right)
}\] 

A p.s.d.\ Gram representation of $\Delta_{1,6}(f(V_{10}/5\backslash 7)=\Delta_{1,6}(\frac{\partial}{\partial x_5}f_{10}|_{x_7=0})$ is $m_2,G_2 :$
\[{\small
\left(
\begin{array}{c}
 x_{9} x_{10} \\
 x_{8} x_{10} \\
 x_{8} x_{9} \\
 x_{4} x_{10} \\
 x_{4} x_{9} \\
 x_{4} x_{8} \\
 x_{3} x_{10} \\
 x_{3} x_{9} \\
 x_{3} x_{8} \\
 x_{2} x_{10} \\
 x_{2} x_{9} \\
 x_{2} x_{8} \\
 x_{2} x_{4} \\
 x_{2} x_{3} \\
\end{array}
\right)
 \ \
\left(
\begin{array}{cccccccccccccc}
 1 & \frac{1}{2} & \frac{1}{2} & \frac{1}{2} & \frac{1}{2} & \frac{1}{3} & \frac{1}{2} &
   \frac{1}{2} & \frac{1}{3} & \frac{1}{2} & \frac{1}{2} & \frac{1}{2} & \frac{1}{2} &
   \frac{1}{2} \medskip \\
 \frac{1}{2} & 1 & \frac{1}{2} & \frac{1}{2} & \frac{1}{3} & \frac{1}{2} & \frac{1}{2} &
   \frac{1}{3} & \frac{1}{2} & \frac{1}{2} & \frac{1}{2} & \frac{1}{2} & \frac{1}{2} &
   \frac{1}{2} \medskip \\
 \frac{1}{2} & \frac{1}{2} & 1 & \frac{1}{3} & \frac{1}{2} & \frac{1}{2} & \frac{1}{3} &
   \frac{1}{2} & \frac{1}{2} & \frac{1}{2} & \frac{1}{2} & \frac{1}{2} & \frac{1}{2} &
   \frac{1}{2} \medskip \\
 \frac{1}{2} & \frac{1}{2} & \frac{1}{3} & 1 & \frac{1}{2} & \frac{1}{2} & \frac{1}{2} &
   \frac{1}{4} & \frac{1}{4} & \frac{1}{2} & \frac{1}{2} & \frac{1}{2} & \frac{1}{2} &
   \frac{1}{2} \medskip \\
 \frac{1}{2} & \frac{1}{3} & \frac{1}{2} & \frac{1}{2} & 1 & \frac{1}{2} & \frac{1}{4} &
   \frac{1}{2} & \frac{1}{4} & \frac{1}{2} & \frac{1}{2} & \frac{1}{2} & \frac{1}{2} &
   \frac{1}{2} \medskip \\
 \frac{1}{3} & \frac{1}{2} & \frac{1}{2} & \frac{1}{2} & \frac{1}{2} & 1 & \frac{1}{4} &
   \frac{1}{4} & \frac{1}{2} & \frac{1}{2} & \frac{1}{2} & \frac{1}{2} & \frac{1}{2} &
   \frac{1}{2} \medskip \\
 \frac{1}{2} & \frac{1}{2} & \frac{1}{3} & \frac{1}{2} & \frac{1}{4} & \frac{1}{4} & 1 &
   \frac{1}{2} & \frac{1}{2} & \frac{1}{2} & \frac{1}{2} & \frac{1}{2} & \frac{1}{2} &
   \frac{1}{2} \medskip \\
 \frac{1}{2} & \frac{1}{3} & \frac{1}{2} & \frac{1}{4} & \frac{1}{2} & \frac{1}{4} &
   \frac{1}{2} & 1 & \frac{1}{2} & \frac{1}{2} & \frac{1}{2} & \frac{1}{2} & \frac{1}{2}
   & \frac{1}{2} \medskip \\
\end{array}
\right)
}\]

A p.s.d.\ Gram representation of $\Delta_{1,7}(f(V_{10}/5)=\Delta_{1,7}(\frac{\partial}{\partial x_5}f_{10})$ is $ m_3, G_3$:
\[{\small
\left(
\begin{array}{c}
 x_{3} x_{10} \\
 x_{6} x_{10} \\
 x_{3} x_{6} \\
 x_{8} x_{10} \\
 x_{3} x_{8} \\
 x_{2} x_{10} \\
 x_{2} x_{3} \\
 x_{2} x_{8} \\
 x_{9} x_{10} \\
 x_{3} x_{9} \\
 x_{6} x_{9} \\
 x_{8} x_{9} \\
 x_{2} x_{9} \\
 x_{4} x_{10} \\
 x_{3} x_{4} \\
 x_{4} x_{6} \\
 x_{4} x_{8} \\
 x_{2} x_{4} \\
 x_{4} x_{9} \\
\end{array}
\right)
\ \ 
\left(
\begin{array}{ccccccccccccccccccc}
 1 & \frac{1}{2} & \frac{1}{2} & \frac{1}{2} & \frac{1}{2} & \frac{1}{2} & \frac{1}{2} &
   \frac{1}{4} & \frac{1}{2} & \frac{1}{2} & \frac{1}{3} & \frac{1}{3} & \frac{1}{3} &
   \frac{1}{2} & \frac{1}{2} & \frac{1}{2} & \frac{3}{8} & \frac{3}{8} & \frac{1}{3} \medskip \\
 \frac{1}{2} & 1 & \frac{1}{2} & \frac{1}{2} & \frac{1}{4} & \frac{1}{2} & \frac{1}{4} &
   0 & \frac{1}{2} & \frac{1}{3} & \frac{1}{2} & \frac{1}{4} & \frac{1}{4} & \frac{1}{2}
   & \frac{1}{2} & \frac{1}{2} & \frac{1}{4} & \frac{1}{4} & \frac{1}{3} \medskip \\
 \frac{1}{2} & \frac{1}{2} & 1 & \frac{1}{4} & \frac{1}{2} & \frac{1}{4} & \frac{1}{2} &
   0 & \frac{1}{3} & \frac{1}{2} & \frac{1}{2} & \frac{1}{4} & \frac{1}{4} & \frac{1}{2}
   & 1 & 1 & \frac{1}{2} & \frac{1}{2} & \frac{1}{2} \medskip \\
 \frac{1}{2} & \frac{1}{2} & \frac{1}{4} & 1 & \frac{1}{2} & \frac{1}{2} & \frac{3}{8} &
   \frac{1}{2} & \frac{1}{2} & \frac{1}{3} & \frac{1}{4} & \frac{1}{2} & \frac{1}{3} &
   \frac{1}{2} & \frac{1}{4} & \frac{1}{4} & \frac{1}{2} & \frac{3}{8} & \frac{1}{3} \medskip \\
 \frac{1}{2} & \frac{1}{4} & \frac{1}{2} & \frac{1}{2} & 1 & \frac{3}{8} & \frac{1}{2} &
   \frac{1}{2} & \frac{1}{3} & \frac{1}{2} & \frac{1}{4} & \frac{1}{2} & \frac{3}{8} &
   \frac{3}{8} & \frac{1}{2} & \frac{1}{2} & \frac{1}{2} & \frac{1}{2} & \frac{3}{8} \medskip \\
 \frac{1}{2} & \frac{1}{2} & \frac{1}{4} & \frac{1}{2} & \frac{3}{8} & 1 & \frac{1}{2} &
   \frac{1}{2} & \frac{1}{2} & \frac{1}{3} & \frac{1}{4} & \frac{1}{3} & \frac{1}{2} &
   \frac{1}{2} & \frac{1}{4} & \frac{1}{4} & \frac{3}{8} & \frac{1}{2} & \frac{1}{3} \medskip \\
 \frac{1}{2} & \frac{1}{4} & \frac{1}{2} & \frac{3}{8} & \frac{1}{2} & \frac{1}{2} & 1 &
   \frac{1}{2} & \frac{1}{3} & \frac{1}{2} & \frac{1}{4} & \frac{3}{8} & \frac{1}{2} &
   \frac{3}{8} & \frac{1}{2} & \frac{1}{2} & \frac{1}{2} & \frac{1}{2} & \frac{3}{8} \medskip \\
\end{array}
\right)
}\]

\vspace{3in} $\;$ \\

A p.s.d.\ Gram representation of $\Delta_{7,9}(f(V_{10}\backslash5)=\Delta_{7,9}(f_{10}|_{x_5=0})$ is $m_4, G_4:$
\right)
}\]

A p.s.d.\ Gram representation of $\Delta_{5,7}(f(V_{10})$ is $m_5, G_5:$\\
\[
{\small 
%
\right)
}
\]

\end{document}